\documentclass[12pt]{article}

\usepackage{latexsym,amsfonts,amsmath,amsthm,amssymb, epsfig}
\usepackage{color}

\usepackage{url}
\usepackage{hyperref}

\usepackage[vcentering,dvips]{geometry}
\newtheorem{theorem}{Theorem}

\newtheorem{proposition}[theorem]{Proposition}

\newtheorem*{oq}{Open Question}
\newtheorem{remark}{Remark}

\theoremstyle{remark}

\theoremstyle{definition}

\newcommand{\R}{\mathbb{R}}

\newcommand{\N}{\mathbb{N}}

\newcommand{\qmc}{{\rm qmc}}
\newcommand{\bx}{{\rm box}}
\newcommand{\di}{{\rm disc}}

\newcommand{\per}{{\rm per}} 
\newcommand{\ex}{{\rm ext}}
\newcommand{\iit}{{\rm int}} 
\newcommand{\rd}{\,{\rm d}}
\newcommand{\bsx}{{\boldsymbol x}}
\newcommand{\bsy}{{\boldsymbol y}}
\newcommand{\bst}{{\boldsymbol t}}
\newcommand{\bsa}{{\boldsymbol a}}
\newcommand{\bszero}{{\boldsymbol 0}}
\newcommand{\bsone}{{\boldsymbol 1}}
\newcommand{\bsdelta}{{\boldsymbol \delta}}
\newcommand{\uu}{\mathfrak{u}}

\newcommand{\cP}{\mathcal{P}}
\newcommand{\cA}{\mathcal{A}}
\newcommand{\cF}{\mathcal{F}}
\newcommand{\DD}{\mathbb{D}}

\title{Tractability versus curse of dimensionality for geometric $L_p$-discrepancies}  

\author{Erich Novak and  Friedrich Pillichshammer}

\date{} 

\begin{document}

\maketitle

\begin{abstract}
This paper studies tractability versus the curse of dimensionality for several geometric $L_p$-discrepancies through a unified discrepancy--integration duality framework, where worst case integration errors in suitable function spaces equal the corresponding discrepancies. A general lower bound method for non-negative linear rules in spaces under broad tensor-product assumptions 
establishes exponential information complexity in the dimension~$d$, yielding the curse of dimensionality for the respective discrepancy. We complement this overview by new results on discrepancy--integration duality and the curse of dimensionality for the periodic $L_p$-discrepancy. The current state of research on this general problem is summarized in a clearly laid out table, which also highlights the remaining open questions. 
\end{abstract}

\centerline{\begin{minipage}[hc]{130mm}{
}
\end{minipage}}

\section{Geometric Discrepancy}\label{sec:intro}

For a set $\cP=\{\bsx_1,\ldots,\bsx_N\}$ consisting of $N$ points in the $d$-dimensional unit-cube $[0,1)^d$ the local discrepancy with respect to a measurable test set $C \subseteq [0,1)^d$ is defined as $$\Delta_{\cP}(C):=\frac{|\{k \in \{1,2,\ldots,N\}\ : \ \bsx_k \in C\}|}{N}-{\rm volume}(C).$$ Choosing a suitable class $\mathcal{C}$ of test sets, that are suitably parametrized, the local discrepancy can be considered as a function of this parameter which is usually emphasized by the nomenclature ``local discrepancy function''. Taking a norm of the local discrepancy functions leads to the notion of a (geometric) discrepancy with respect to the given class $\mathcal{C}$ of test sets. This yields a quantitative measure for the uniformity of the point set $\cP$ with respect to $\mathcal{C}$. The term ``geometric'' appears in \cite[Sections~1.2 and 1.3]{Mat99} and refers to underlying geometric classes $\mathcal{C}$ of test sets.

Many notions of discrepancy are based on a suitable $L_p$-norm. There are a lot of examples in literature. In the present paper we consider the prominent notions of anchored/quadrant $L_p$-discrepancy, a concept that contains the star- and centered discrepancies as special cases, and unanchored $L_p$-discrepancy, such as the extreme- and periodic $L_p$-discrepancy.

In applications (numerical analysis) one usually studies a slightly more general concept of discrepancy with nodes weighted by reals $\cA=\{a_1,\ldots,a_N\}$. In this sense the local discrepancy is defined as $$\Delta_{\cP,\cA}(C):=\sum_{k=1}^N a_k {\bf 1}_C(\bsx_k)-{\rm volume}(C).$$ If $a_1=\ldots=a_N=1/N$, we speak about ``QMC weights'' and denote the corresponding set by $\cA^{\qmc}$, then we are in the classical case, i.e., $\Delta_{\cP,\cA^{\qmc}}(C)=\Delta_{\cP}(C)$. 

Let $D_N(\cP,\cA)=\|\Delta_{\cP,\cA}\|$ be any (geometric) discrepancy, where the norm is not yet specified (often it is an $L_p$ norm). Then for $d,N \in \N$ the $N$-th minimal discrepancy in dimension $d$ is defined as $${\rm disc}(N,d):=\inf_{\cP,\cA} D_N(\cP,\cA),$$ where the infimum is extended over all $N$-element sets $\cP$ in $[0,1)^d$ and any real weights $\cA$. We write ${\rm disc}^+$, if the infimum is extended over all $N$-element sets $\cP$ but exclusively over non-negative weights $\cA^+$, and ${\rm disc}^{\qmc}$ if we allow only QMC weights and extend the infimum only over $N$-point sets $\cP$.

Define further ${\rm disc}(0,d):=\|{\rm Vol}\|$ as the norm of the volume functional ${\rm Vol}: C \mapsto {\rm volume}(C)$. We interpret this quantity as the discrepancy of the empty set and call it therefore the ``initial discrepancy''. The inverse of discrepancy is defined as the minimal number of points that is needed in order to reduce the initial discrepancy by a factor of $\varepsilon$. In more detail, for $\varepsilon \in (0,1)$ we define $$N_{\di}^{\bullet}(\varepsilon,d):=\min\{N \in \N : {\rm disc}^{\bullet}(N,d) \le \varepsilon \, {\rm disc}(0,d)\},$$ where $\bullet \in \{\text{blank},+,\qmc\}$. 
A central  question  is  the growth rate of $N_{\di}^{\bullet}(\varepsilon,d)$ when $\varepsilon \rightarrow 0^+$ 
and/or $d \rightarrow \infty$ (i.e., when $d+\varepsilon^{-1} \rightarrow \infty$). Such questions are typically studied in the field ``Information Based Complexity'' (see \cite{NW08,NW10,NW12}). 

If $N_{\di}^{\bullet}(\varepsilon,d)$ grows at least exponentially fast in $d$ for a fixed $\varepsilon >0$, 
more precisely, if there exists a real $C>1$ such that $N_{\di}^{\bullet}(\varepsilon,d) \ge C^d$ 
for infinitely many $d \in \N$, then the discrepancy (for arbitrary weights or non-negative weights or QMC-weights,  respectively, depending on $\bullet \in \{\text{blank},+,\qmc\}$) is said to suffer from the curse of dimensionality. Of course, this is a very unpleasant behavior. Naturally, one would prefer a subexponential growth rate of $N_{\di}^{\bullet}(\varepsilon,d)$ in $d$ and $\varepsilon^{-1}$ for $d+\varepsilon^{-1} \rightarrow \infty$. In this case a discrepancy is said to be tractable and there are various notions of tractability in order to classify the growth rate of $N_{\di}^{\bullet}(\varepsilon,d)$ more accurately. For example, polynomially tractability means that there exist numbers $C,\sigma>0$ and $\tau \ge 0$ such that $N_{\di}^{\bullet}(\varepsilon,d)\le C d^{\tau} \varepsilon^{- \sigma}$ for all $\varepsilon \in (0,1)$ and $d \in \N$. If this holds even for $\tau=0$, then one speaks about strong polynomial tractability.

To prove for a discrepancy the curse of dimensionality, one needs general lower bounds for arbitrary sets of points and weights. Conversely, to prove tractability, a good upper bound for specific sets of points and weights suffices. Recently, we have developed a useful method to prove the curse of dimensionality for various geometric $L_p$-discrepancies (see \cite{NP25c,NP25b,NP25,NP26a}). This method, which is based on the duality between discrepancy and integration, will be presented in a general context in Section~\ref{sec:meth}. But first (Section~\ref{sec:diintdu}) we explain the framework of a ``discrepancy--integration duality''. In Section~\ref{sec:exa} we present several concrete examples.  Most of the results are scattered across various publications. Here we provide a systematic overview with precise references or proofs, if no reference is available. In the final Section~\ref{sec:close}, we summarize everything again in a clearly laid out table, which also clearly shows the remaining open questions.

\section{Discrepancy--integration duality}\label{sec:diintdu}

Often a discrepancy is related to the worst-case error of a quasi-Monte Carlo (QMC) integration problem in the sense that there exists a suitable normed space of functions over $[0,1]^d$ such that the worst-case error of a QMC rule for this space is exactly a suitable (geometric) discrepancy. Historically, the famous Koksma-Hlawka inequality was the first result in this direction (see \cite{Hla0,Kok42}). We call such a framework a ``discrepancy--integration duality'' and we call the corresponding integration problem a dual integration problem. 

We consider linear rules based on a point set $\cP=\{\bsx_1,\ldots,\bsx_N\}$ in $[0,1)^d$ and corresponding real (integration) weights $\cA=\{a_1,\ldots,a_N\}$ of the form   
\begin{equation}\label{lin:rule}
Q_{\cP,\cA}(f):=\sum_{k=1}^N a_k f(\bsx_k) \quad \mbox{for functions $f:[0,1]^d \rightarrow \R$.}
\end{equation}
Sometimes we restrict ourselves to non-negative weights $a_k > 0$ for all $k \in \{1,\ldots,N\}$, which is indicated by the notation $\cA^+$. A QMC rule is the special case of $a_k=1/N$ for all $k \in \{1,\ldots,N\}$. We indicate this by the notation $\cA^{\qmc}$

Assume there exists a Banach space $(\cF_d,\|\cdot\|_d)$ of measurable functions over $[0,1]^d$ for which we consider integration 
$$I(f):=\int_{[0,1]^d} f(\bsx)\rd \bsx \qquad \mbox{for } f \in \cF_d$$ by means of linear rules or QMC rules $Q_{\cP,\cA}$, such that the  worst-case error $$e(Q_{\cP,\cA};\cF_d):=\sup_{f \in \cF_d \atop \|f\|_d \le 1}|I(f)-Q_{\cP,\cA}(f)|$$ satisfies $$e(Q_{\cP,\cA};\cF_d) = D_N(\cP,\cA).$$
Then we call this a ``discrepancy--integration duality'' and we call integration in $\cF_d$ the dual integration problem to the discrepancy $D_N$. 

A discrepancy--integration duality assigns an important interpretation to the discrepancy beyond its purely geometric definition, thereby enabling important applications in numerical analysis.  A detailed discussion for the $L_2$ case can be found in \cite[Section~9.6]{NW10}. On the other hand, a discrepancy--integration duality is also a very useful tool for making theoretical statements about the discrepancy through the study of the integration problem itself. This approach has been pursued in several papers, one particular example is \cite[Proof of Theorem~5]{HKP}.  Here we present our most recent approach to proving lower bounds for worst-case integration errors/discrepancies, which has led to new insights into the curse of dimensionality for various geometric $L_p$-discrepancies.

\section{Lower bound method}\label{sec:meth}

For an integration problem over $(\cF_d,\|\cdot\|_d)$ the $N$-th minimal error is defined as $$e(N,d)=e(N,d;\cF_d):=\min_{\cP,\cA}  e(Q_{\cP,\cA};\cF_d),$$ where the minimum is extended over all $N$-element sets $\cP$ in $[0,1)^d$ with corresponding weights $\cA$. Furthermore, if we restrict to non-negative weights we write $$e^+(N,d)=e^+(N,d;\cF_d):=\min_{\cP,\cA^+}  e(Q_{\cP,\cA^+};\cF_d),$$ and in the QMC-case we write $$e^{\qmc}(N,d)=e^{\qmc}(N,d;\cF_d):=\min_{\cP}  e(Q_{\cP,\cA^{\qmc}};\cF_d).$$ The initial error is $$e(0,d)=e(0,d;\cF_d):=\sup_{f \in \cF_d \atop \|f\|_d \le 1}|I(f)|.$$ Furthermore, we define the information complexity as $$N_{\iit}^{\bullet}(\varepsilon,d):=\min\{N \in \N : e^{\bullet}(N,d) \le \varepsilon \, e(0,d)\} \quad \mbox{for } \varepsilon \in (0,1),\ d \in \N,$$ where $\bullet \in \{{\rm blank},+,\qmc\}$. Obviously, $N_{\iit}(\varepsilon,d) \le N_{\iit}^+(\varepsilon,d) \le N_{\iit}^{\qmc}(\varepsilon,d)$.

\begin{theorem}\label{thm1}
Assume that the following holds:
\begin{enumerate}
\item For functions $f \in \cF_d$ of the form $f(\bsx)=f_1(x_1) \cdots f_d(x_d)$ we have $\|f\|_d \le \|f_1\|_1 \cdots \|f_d\|_1$, and
\item $e(0,d)=e(0,1)^d$.
\item There exists a non-negative function $h \in \cF_1$, such that $I(h/\|h\|_1)=e(0,1)$.
\item There exists for every $y \in [0,1]$ a function $s_y \in \cF_1$ such that $s_y \ge 0$, $s_y(y)=h(y)$, and $$A:=\inf_{y \in [0,1]} \frac{I(h)}{I(s_y)} > 1 \quad \mbox{ and } \quad B:= \inf_{y \in [0,1]} \frac{\|h\|_1}{\|s_y\|_1} >1.$$
\end{enumerate}
Then, for every $d \in \N$ and every $\varepsilon \in (0,\tfrac{1}{2}]$ we have $$N_{\iit}^+(\varepsilon,d) \ge C^d (1-2 \varepsilon), $$ where $C:= \min(A,B)>1$. In particular, the integration problem (for non-negative weigths and especially for QMC weights) suffers from the curse of dimensionality.
\end{theorem}

\begin{remark}\rm
Assumptions~1 and 2 are readily satisfied for tensor-product problems. The function $h$ from Assumption~3 is usually referred to as {\it worst-case function} since its integral normalized by the norm of $h$ equals the initial error in dimension $d=1$. Together with Assumption~2 this implies that for $h_d(\bsx):=h(x_1)\cdots h(x_d)$ we have $I(h_d/\|h\|_1^d)=e(0,d)$. The functions $s_y$ 
from Assumption~4 can be considered as ``splines'': They have small integrals and norms, but fulfill $s_y(y)=h(y)$. 
\end{remark}

\begin{proof}[Proof of Theorem~\ref{thm1}]
Consider a linear algorithm $Q_{\cP,\cA}$ of the form \eqref{lin:rule} based in nodes $\cP=\{\bsx_1,\ldots,\bsx_N\}$ in $[0,1]^d$ and with non-negative weights $\cA^+=\{a_1,\ldots,a_N\}$. For $k \in \{1,\ldots,N\}$ and $j \in \{1,\ldots,d\}$ let $x_{k,j}$ be the $j$-th coordinate of the point $\bsx_k$. For $k \in \{1,\ldots,N\}$ we define functions 
$$
P_k(\bsx) := s_{x_{k,1}}(x_1) s_{x_{k,2}}(x_2) \cdots s_{x_{k,d}}(x_d),\quad \mbox{for $\bsx=(x_1,\ldots,x_d)\in [0,1]^d$.}  
$$

Consider the two functions $h_d(\bsx):= h(x_1) \cdots h(x_d)$ and $f^*(\bsx):= \sum_{i=1}^N P_i(\bsx)$.
Since $Q_{\cP,\cA^+}$ uses only non-negative weights we have 
\begin{align*}
Q_{\cP,\cA^+}(f^*) = & \sum_{k=1}^N a_k \sum_{j=1}^N P_j(\bsx_k) \ge \sum_{k=1}^N a_k P_k(\bsx_k)
=  \sum_{k=1}^N a_k s_{x_{k,1}}(x_{k,1})\cdots s_{x_{k,d}}(x_{k,d}) \\
= & \sum_{k=1}^N a_k h(x_{k,1}) \cdots h(x_{k,d}) =  \sum_{k=1}^N a_k h_d(\bsx_k) = Q_{\cP,\cA^+}(h_d).
\end{align*}

Now for real $y$ we use the notation $(y)_+:= \max (y, 0)$. Then we have  
\begin{equation}\label{errest1}
e(Q_{\cP,\cA^+};\cF_d)  \ge \frac{(I(h_d) - I(f^*))_+}{2 \max ( \|h\|_1^d, \| f^*\|_{d})},
\end{equation}
which is trivially true if $I(h_d) \le  I(f^*)$ and which is easily shown if $I(h_d) >  I(f^*)$, because then 
\begin{align*}
(I(h_d) - I(f^*))_+ \le & I(h_d) - Q_{\cP,\cA^+}(h_d)+Q_{\cP,\cA^+}(f^*)-I(f^*)\\
\le & \|h_d\|_{d} \, e(Q_{\cP,\cA^+};\cF_{d})+\|f^*\|_{d} \, e(Q_{\cP,\cA^+};\cF_{d})\\
\le & 2 \max(\|h\|_1^d,\|f^*\|_{d}) \, e(Q_{\cP,\cA^+};\cF_{d}),
\end{align*}
where we used that $\|h_d\|_{d} \le \|h\|_1^d$. This implies \eqref{errest1}.

From the triangle inequality we obtain 
$$
\| f^* \|_{d} \le \frac{N \, \|h\|_1^d}{B^d} \qquad \mbox{and} \qquad I(f^*) \le N \left (\frac{I(h)}{A}\right)^d.
$$ 
Inserting into \eqref{errest1} yields 
$$
e(Q_{\cP,\cA^+};\cF_d) \ge \frac{ I(h)^d  (1- N A^{-d})_+} {2 \|h\|_1^d \max (1, 
N B^{-d}) } = \frac{e(0,d) (1- N A^{-d})_+} {2 \max (1, 
N B^{-d}) },
$$
where we used that $(I(h)/\|h\|_1)^d=I(h/\|h\|_1)^d= e(0,1)^d = e(0,d)$.
This yields  
\begin{equation*}
e^+(N,d)\ge \frac{ e(0,d) (1- N A^{-d})_+} {2 \max (1, 
N B^{-d}) }.
\end{equation*}

Now let $\varepsilon \in (0,1/2)$ and assume that $e^+(N,d) \le \varepsilon \, e(0,d)$. This implies that $$2 \varepsilon  \max (1, N B^{-d})  \ge  (1 - N A^{-d})_+.$$

Assumption~4 yields $C:=\min(A,B) >1$. If $N \le C^d$, then we obtain
\begin{align*}
1- N A^{-d}  = & (1 - N A^{-d})_+ \le 2 \varepsilon \max (1, N B^{-d}) =  2 \varepsilon. 
\end{align*}
Hence $N \ge A^d (1-2 \varepsilon) \ge C^d (1-2 \varepsilon)$. If $N \ge C^d$, then we trivially have $N \ge C^d (1-2 \varepsilon)$. This yields $N_{\iit}^+(\varepsilon,d)\ge C^d (1-2 \varepsilon)$, as claimed. \qed
\end{proof}

\section{Examples}\label{sec:exa}

We present several important examples of $L_p$-discrepancies where the method from Sec.~\ref{sec:meth} was successfully applied. In the case of periodic $L_p$-discrepancy, we can also add a new result. Throughout, let $p,q \in [1,\infty]$ be H\"older conjugates, i.e., $\frac1p+\frac1q=1$.

\subsection{Anchored discrepancies} 

A general representation of the anchored discrepancies can best be introduced using the concept of quadrant discrepancy, 
which uses the following class of test sets: Fix $a \in [0,1]$. For $t \in [0,1]$ define the intervals 
$$Q(t)=Q_a(t):=\left\{
\begin{array}{ll}
[0,t) & \mbox{if } t \le a,\\
{[}t,1) & \mbox{if } t > a. 
\end{array}
\right.$$
Since $a$ is fixed we will suppress the index $a$ in the notation. In the multidimensional case, for $\bst=(t_1,\ldots,t_d) \in [0,1]^d$ we define $Q(\bst):=Q(t_1)\times \ldots \times Q(t_d)$. Obviously, the volume is $\lambda(Q(\bst))=\prod_{j=1}^d (t_j\, {\bf 1}_{[0,a]}(t_j)+(1-t_j) \, {\bf 1}_{(a,1]}(t_j))$. Using as test sets the class $\mathcal{C}=\{Q(\bst) : \bst \in [0,1]^d\}$, the $a$-quadrant $L_p$-discrepancy is defined as $$L_{p,N}^{\boxplus_a}(\cP,\cA):=\|\Delta_{\cP,\cA}(Q(\cdot))\|_{L_p([0,1]^d)} \quad \mbox{ for } p \in [1, \infty].$$
The initial $a$-quadrant $L_p$-discrepancy equals $${\rm disc}_p^{\boxplus_a}(0,d)=\left\{ 
\begin{array}{ll}
\left(\frac{a^{p+1}+(1-a)^{p+1}}{p+1}\right)^{d/p} & \mbox{for } p \in [1,\infty),\\[0.5em]
\max(a,1-a)^d & \mbox{for } p=\infty.
\end{array}
\right.$$ 
Note that the initial $a$-quadrant $L_p$-discrepancy is exponentially small in the dimension $d$ for almost all instances. The only exception is the case $p=\infty$, $a \in \{0,1\}$, where the initial discrepancy equals 1.

The general definition of quadrant discrepancy includes two particularly important special cases:

\begin{itemize}
\item {\bf Star discrepancy:} For $a=1$ we obtain the classical star $L_p$-discrepancy, which is the most prominent notion in discrepancy theory. We denote it by $L_{p,N}^{\ast}(\cP,\cA)$ for $N$-element sets $\cP$ with corresponding weights $\cA$. 
Obviously, in the limiting case $a=1$ one would hardly speak of a quadrant, but purely formally the star $L_p$-discrepancy falls under the term quadrant $L_p$-discrepancy.
 
\item {\bf Centered discrepancy:} For $a=\tfrac{1}{2}$ the notion of centered $L_p$-discrepancy is covered. For $p=2$ centered $L_2$-discrepancy was studied by Hickernell~\cite{hick98}. See also \cite{NW01} and  \cite[Sec.~9.8.2 and 11.4.3]{NW10}.  
\end{itemize} 

The general definition of quadrant discrepancy is closely related to another discrepancy concept, which is also categorized under the umbrella term of anchored discrepancy (see \cite[Example~1]{HW12} and \cite[Section~9.5.3]{NW10}) and which uses so-called anchored intervals as test sets. Again, fix $a \in [0,1]$, the anchor. A one-dimensional, in $a$ anchored interval is defined as the set $J(t):=[\min(t,a),\max(t,a))$ for $t \in [0,1]$. In the multidimensional case, for $\bst=(t_1,\ldots,t_d) \in [0,1]^d$ define $J(\bst) :=J(t_1)\times \ldots \times J(t_d)$. Obviously, 
the volume is $\lambda(J(\bst))=\prod_{j=1}^d|t_j-a|$. Using as test sets the class $\mathcal{C}=\{J(\bst) : \bst \in [0,1]^d\}$, the $a$-anchored $L_p$-discrepancy of $\cP$ is defined as $$L_{p,N}^{\pitchfork_a}(\cP,\cA):=\|\Delta_{\cP,\cA}(J(\cdot))\|_{L_p([0,1]^d)} \quad \mbox{ for } p \in [1, \infty].$$

The initial $a$-anchored $L_p$-discrepancy equals $${\rm disc}_p^{\pitchfork_a}(0,d)=\left\{ 
\begin{array}{ll}
\left(\frac{a^{p+1}+(1-a)^{p+1}}{p+1}\right)^{d/p} & \mbox{for } p \in [1,\infty),\\[0.5em]
\max(a,1-a)^d & \mbox{for } p=\infty,
\end{array}
\right.$$ 
which coincides with the initial $a$-quadrant $L_p$-discrepancy.
 
Obviously, the end point $a=0$ recovers the star $L_p$-discrepancy and $a=1$ the $0$-quadrant $L_p$-discrepancy. In general, the close relationship with quadrant discrepancy is described by the following proposition. 
The presented result in its form and generality is new. The special case $p=2$ is implicitly contained in \cite[Sections~9.5.3 and 9.5.4]{NW10} and derived there by means of a discrepancy--integration duality with regard to integration in reproducing kernel Hilbert spaces.

\begin{proposition}\label{pr2}
For $a \in [0,1]$ and a point set $\cP$ in $[0,1)^d$ define $\overline{\cP}_a:=\{\bsa-\bsx \pmod{1} : \bsx \in \cP\}$, where $\bsa=(a,\ldots,a)$ (vector of length $d$). Then we have $$L_{p,N}^{\pitchfork_a}(\cP,\cA)= L_{p,N}^{\boxplus_a}(\overline{\cP}_a,\cA).$$ In particular
\begin{equation*}
{\rm disc}_p^{\pitchfork_a}(N,d)={\rm disc}_p^{\boxplus_a}(N,d).
\end{equation*}
\end{proposition}
  
\begin{proof}
To avoid confusing notation problems, we will restrict ourselves to the one-dimensional case. For $x \in [0,1]$ we write $\overline{x}:=a-x \pmod{1}$. Then for any point $x_k\in \cP$ we have $x_k \in J(t)$ if and only if $\overline{x}_k \in Q(\overline{t})$.  Furthermore, $\lambda(J(t))=\lambda(Q(\overline{t}))$. This yields $\Delta_{\cP,\cA}(J(t))=\Delta_{\cP,\cA}(Q(\overline{t}))$, from which we obtain the result. \qed
\end{proof}

Proposition~\ref{pr2} tells us that all statements about the 
$N$-th minimal quadrant discrepancy (such as curse/tractability) apply directly to the $N$-th minimal anchored discrepancy and vice versa. Furthermore, the discrepancy--integration duality from the forthcoming Theorem~\ref{thm2} can be easily re-written in terms of $a$-anchored $L_p$-discrepancy.\\

\noindent{\bf Discrepancy--integration duality:} For $q \in [1,\infty]$ and $a \in [0,1]$ let 
$$
W_{a,q}^1([0,1]):=\{f:[0,1] \rightarrow \R \ : \ f\, \mbox{ abs. cont., $f(a)=0$ and}\, f' \in L_q([0,1])\},
$$
equipped with the norm $\|f\|_{1,q}:=\|f'\|_{L_q}$. The $d$-fold tensor product space is $$W_{a,q}^{\bsone}:=W_{a,q}^1([0,1]) \otimes \cdots \otimes W_{a,q}^1([0,1])$$ equipped with the crossnorm $\|f\|_{d,q}:=\|f^{(1,1,\ldots,1)}\|_{L_q}$. Then a discrepancy--integration duality is given in the following way:

\begin{theorem}\label{thm2}
We have $$e(Q_{\cP,\cA};W_{a,q}^{\bsone})=L_{p,N}^{\boxplus_a}(\cP,\cA).$$ 
In particular
\begin{equation}\label{eq:did:anch}
e(N,d;W_{a,q}^{\bsone})={\rm disc}_p^{\boxplus_a}(N,d).
\end{equation}
\end{theorem} 

A proof for general $a$ and $p=2$ can be found in \cite[Sec.~9.5.4]{NW10}. For general $p$ and $a =1$ see \cite[Sec.~9.8.1]{NW10} and for general $p$ and $a=\tfrac{1}{2}$ see \cite{NW01} or \cite[Sec.~11.4.3]{NW10}. Nevertheless, we nowhere found a complete unifying proof of the general setting stated above in the literature. Therefore we add the proof here.

\begin{proof}[Proof of Theorem~\ref{thm2}]
For a vector $\bsx=(x_1,\ldots,x_d) \in [0,1]^d$ let $\uu \subseteq \{1,\ldots,d\}=:[d]$ be the set of indices $j$ for which $x_j \le a$. Hence, for $j \in \uu^c$, where $\uu^c:=[d]\setminus \uu$, we have $x_j >a$. We write $\bsx_{\uu}$ for the projection of the vector $\bsx$ to the coordinates whose indices belong to $\uu$. Furthermore, let $\bsa=(a,\ldots,a)$ (length $d$).  Then $$f(\bsx)=\int_{[\bsx_\uu,\bsa_\uu]} \int_{[\bsa_{\uu^c},\bsx_{\uu^c}]} f^{(1,\ldots,1)}(\bst) \rd \bst.$$ 

Let $\mu:=\lambda-\sum_{k=1}^N a_k \delta_{\bsx_k}$ where $\delta_{\bsx_k}$ is the Dirac measure concentrated in $\bsx_k$ and $\lambda$ the Lebesgue measure on $[0,1]^d$ and let $L(f):= I(f)-Q_{\cP,\cA}(f)$. Then we have
\begin{eqnarray*}
L(f) & = & \int_{[0,1]^d} f(\bsx) \rd \mu(\bsx) = \sum_{\uu \subseteq [d]} \int_{[\bszero_\uu,\bsa_\uu]} \int_{[\bsa_{\uu^c},\bsone_{\uu^c}]} f(\bsx) \rd \mu(\bsx)\\
& = & \sum_{\uu \subseteq [d]} \int_{[\bszero_\uu,\bsa_\uu]} \int_{[\bsa_{\uu^c},\bsone_{\uu^c}]} \int_{[\bsx_\uu,\bsa_\uu]} \int_{[\bsa_{\uu^c},\bsx_{\uu^c}]} f^{(1,\ldots,1)}(\bst) \rd \bst \rd \mu(\bsx)\\
& = & \sum_{\uu \subseteq [d]} \int_{[\bszero_\uu,\bsa_\uu]} \int_{[\bsa_{\uu^c},\bsone_{\uu^c}]}  f^{(1,\ldots,1)}(\bst) \int_{[\bszero_\uu,\bst_\uu]} \int_{[\bst_{\uu^c},\bsone_{\uu^c}]}  \rd \mu(\bsx) \rd \bst\\
& = & \sum_{\uu \subseteq [d]} \int_{[\bszero_\uu,\bsa_\uu]} \int_{[\bsa_{\uu^c},\bsone_{\uu^c}]}  f^{(1,\ldots,1)}(\bst) \int_{Q(\bst)}  \rd \mu(\bsx) \rd \bst\\
& = & \int_{[0,1]^d} f^{(1,\ldots,1)}(\bst) \mu(Q(t)) \rd \bst.
\end{eqnarray*}
Hence $$|L(f)| \le \|f^{(1,\ldots,1)}\|_{L_q} \|\mu(Q(\cdot))\|_{L_p} = \|f\|_{d,q} L_{p,N}^{\boxplus_a}(\cP,\cA),$$ because 
$$\mu(Q(\bst))=\lambda(Q(\bst))-\sum_{k=1}^N a_k {\bf 1}_{Q(\bst)}(\bsx_k)=-\Delta_{\cP,\cA}(Q(\bst))$$ and hence $$\|\mu(Q(\cdot))\|_{L_p} = \|\Delta_{\cP,\cA}(Q(\cdot))\|_{L_p}=L_{p,N}^{\boxplus_a}(\cP,\cA).$$ This yields $$e(Q_{\cP,\cA};W_{a,q}^{\bsone}) \le L_{p,N}^{\boxplus_a}(\cP,\cA).$$ 

In order to prove equality let first $p,q \in (1,\infty)$. Set $$f_\ast(\bsx):=-\frac{1}{\|\Delta_{\cP,\cA}(Q(\cdot))\|_{L_p}^{p-1}} \int_{[\bsa,\bsx]} |\Delta_{\cP,\cA}(Q(\bst))|^{p-2} \Delta_{\cP,\cA}(Q(\bst)) \rd \bst.$$ Then we have $f_\ast(\bsx)=0$ whenever at least one component of $\bsx$ equals $a$, and $f_{\ast}^{(1,\ldots,1)}(\bsx) =- \frac{|\Delta_{\cP,\cA}(Q(\bsx))|^{p-2} \Delta_{\cP,\cA}(Q(\bsx))}{\|\Delta_{\cP,\cA}(Q(\cdot))\|_{L_p}^{p-1}}$. Hence 
\begin{eqnarray*}
L(f_\ast)  & = & -\int_{[0,1]^d} f_{\ast}^{(1,\ldots,1)}(\bst) \Delta_{\cP,\cA}(Q(\bst))\rd \bst \\
& = & \frac{1}{\|\Delta_{\cP,\cA}(Q(\cdot))\|_{L_p}^{p-1}}\int_{[0,1]^d} |\Delta_{\cP,\cA}(Q(\bst))|^p \rd \bst\\
& = & \frac{\|\Delta_{\cP,\cA}(Q(\cdot))\|_{L_p}^p}{\|\Delta_{\cP,\cA}(Q(\cdot))\|_{L_p}^{p-1}}=\|\Delta_{\cP,\cA}(Q(\cdot))\|_{L_p} = L_{p,N}^{\boxplus_a}(\cP,\cA).
\end{eqnarray*}
Furthermore, using $(p-1)q=p$ for H\"older conjugates $p,q \in (1,\infty)$,
\begin{eqnarray*}
\|f_{\ast}\|_{d,q}^q & = & \| f_{\ast}^{(1,\ldots,1)} \|_{L_q}^q = \int_{[0,1]^d} \frac{|\Delta_{\cP,\cA}(Q(\bst))|^{(p-1) q}}{\|\Delta_{\cP,\cA}(Q(\cdot))\|_{L_p}^{(p-1) q}} \rd \bst\\
& = & \frac{1}{\|\Delta_{\cP,\cA}(Q(\cdot))\|_{L_p}^p} \|\Delta_{\cP,\cA}(Q(\cdot))\|_{L_p}^p = 1.
\end{eqnarray*}
This yields $e(Q_{\cP,\cA};W_{a,q}^{\bsone}) \ge L_{p,N}^{\boxplus_a}(\cP,\cA)$. Together with the upper bound we obtain the desired result.

For $p=1$, $q=\infty$ choose $f_{\ast}(x):=-\int_{[\bsa,\bsx]} {\rm sgn} ( \Delta_{\cP,\cA}(Q(\bst)) ) \rd \bst$. Then, again $L(f_{\ast})=L_{1,N}^{\boxplus_a}(\cP,\cA)$ and $\|f_{\ast}\|_{1,\infty}=1$ such that the result follows in the same way as above.

For $p=\infty$, $q=1$ and $\varepsilon >0$ set $$E_{\varepsilon}:=\left\{\bst \in [0,1]^d : | \Delta_{\cP,\cA}(Q(\bst))| \ge L_{\infty,N}^{\boxplus_a}(\cP,\cA) - \varepsilon\right\}$$ and $$f_{\ast}(\bsx):= -\int_{[\bsa,\bsx]} \frac{{\rm sgn}(\Delta_{\cP,\cA}(Q(\bst))) {\bf 1}_{E_{\varepsilon}}(\bst)}{\lambda(E_{\varepsilon})} \rd \bst.$$ Then, $f_{\ast}(\bsx)=0$ whenever at least one component of $\bsx$ equals $a$, and $f_{\ast}^{(1,\ldots,1)}(\bsx)
- \frac{{\rm sgn}(\Delta_{\cP,\cA}(Q(\bsx))) {\bf 1}_{E_{\varepsilon}}(\bsx)}{\lambda(E_{\varepsilon})}$, which yields $\|f_{\ast}\|_{d,1} = \|f_{\ast}^{(1,\ldots,1)}\|_{L_1}=1$ and $$L(f_{\ast})= \int_{[0,1]^d} \frac{ {\bf 1}_{E_{\varepsilon}}(\bst)}{\lambda(E_{\varepsilon})} | \Delta_{\cP,\cA}(Q(\bst))| \rd \bst \ge L_{\infty,N}^{\boxplus_a}(\cP,\cA) - \varepsilon.$$  Letting $\varepsilon \rightarrow 0$ yields $e(Q_{\cP,\cA};W_{a,1}^{\bsone}) \ge L_{\infty,N}^{\boxplus_a}(\cP,\cA)$ and this finishes the proof. \qed
\end{proof}
 
\noindent{\bf Tractablility vs. curse of dimensionality:} We distinguish the cases $a \in \{0,1\}$ and $a \in (0,1)$. 
\begin{description}
\item[$a=1$:] This is the case of classical star $L_p$-discrepancy. For $p \in (1,\infty)$ the star $L_p$-discrepancy for non-negative weights $\cA^+$ suffers from the curse of dimensionality. This was shown in \cite{NP25} with the method from Theorem~\ref{thm1} based on the worst-case function $h(x)=1-(1-x)^p$. For $p=2$ this was already shown much earlier by Wo\'{z}niakowski~\cite{Wo99} (nowadays this paper, together with \cite{hnww} could be seen as the initiation of the whole story considered in the present paper). The case $a=0$ is almost identical.

The case $p=1$ is still open.  

For $p=\infty$ the star $L_{\infty}$-discrepancy is polynomially tractable as shown first in a 
paper by Heinrich, Novak, Wasilkowski and Wo\'{z}niakow\-ski~\cite{hnww} (the result even holds for QMC weights $\cA^{\qmc}$). More precisely there exists a $C>0$ such that 
\begin{equation}\label{bd:HNWW}
N_{\infty}^{\ast,\qmc}(\varepsilon,d) \le C d \varepsilon^{-2} \quad \mbox{ for all $d \in \mathbb{N}$ and $\varepsilon \in (0,1)$.}
\end{equation}
The linear dependence in $d$ cannot be improved, see \cite[Theorem~8]{hnww} or \cite[Theorem~1]{Hi04} for details. 
One may choose $C=6.067\ldots$, see \cite{Wei26}.

\item[$a \in (0,1)$:] For $p \in [1,\infty)$ the $a$-anchored $L_p$-discrepancy suffers from the curse of dimensionality (here the case $p=1$ is included). This was shown in \cite[Corollary~2]{NP25b}, where the proof is based on the worst-case function $h(x)={\bf 1}_{[0,a]}(x) (a^p-x^p) +{\bf 1}_{[a,1]}(x)((1-a)^p-(1-x)^p)$. 

For $p=\infty$ and $a=\tfrac{1}{2}$ the $\frac{1}{2}$-quadrant (=centered) $L_{\infty}$-discrepancy suffers also from the curse of dimensionality. This follows immediately from the fact that for $N$-element sets $\cP$ with $N <2^d$ at least one of the $2^d$ quadrants of the unit-cube $[0,1)^d$ remains empty and hence $\di_{\infty}^{\boxplus_{1/2}}(N,d)= 2^{-d} = \di_{\infty}^{\boxplus_{1/2}}(0,d)$ (this is already pointed out in \cite[p.~406]{NW01} and \cite[p.~177]{NW10}). 

The case $p=\infty$ but $a \not\in \{0,\frac{1}{2},1\}$ was still open, see the following remark.
\end{description}

\begin{remark}\rm
Here we consider the ``intermediate case'' $p=\infty$ and $a \in (\frac12,1)$, where the problem is not polynomially tractable 
but there is also no curse since the complexity is polynomial for each fixed 
$\varepsilon >0$; of course the case $a \in (0,\frac12)$ is identical. 

We start with the lower bound and use the discrepancy-integration duality from Theorem~\ref{thm2}, i.e., we consider integration in $W_{a,1}^{\bsone}$. 
For $d=1$ we can define two fooling functions in $W_{a,1}^1$ with disjoint supports 
and integrals arbitrarily close to $a$ and $1-a$, i.e., the normalized integrals are close to 1 and $\frac{1-a}{a}$, 
respectively, since the initial error is exactly $a$ in the present case. 
For $d>1$ we take tensor products and, for $\varepsilon_k = \left(\frac{1-a}{a}\right)^k$ and $d \ge k$, we obtain 
${d \choose k}$ functions with disjoint supports and integrals close to $\varepsilon_k$, 
hence $N_{\infty}^{\boxplus_a}(\varepsilon_k,d)=N_{{\rm int}}^{W_{a,1}^{\bsone}}(\varepsilon_k, d) \ge {d \choose k}$, 
since for any $N$-element set in $[0,1)^d$ with less then ${d \choose k}$ points, the support of at least one of the functions contains no point of the point set. 
It follows that the integration problem, and therefore the $L_{\infty}^{\boxplus_a}$-discrepancy  is not polynomially tractable. 

Now we consider the upper bound for a given $d$ and normalized error $\varepsilon_k = \left(\frac{1-a}{a}\right)^k$.
By the break point $a$ we have, in dimension $d$, altogether $2^d$ subproblems, defined on subcubes $\prod_{j \in \uu}[a,1) \times \prod_{j \in [d]\setminus \uu} [0,a)$ of $[0,1)^d$ for $\uu \subseteq [d]$ of volume $(1-a)^{|\uu|} a^{d-|\uu|}$ whose normalization is $\left(\frac{1-a}{a}\right)^{|\uu|}$. We have to consider only those problems with initial error at least $\varepsilon_k$, i.e., those $\uu \subseteq [d]$ for which $|\uu| \le k$. The number of these problems is $\sum_{\ell=0}^k {d \choose \ell}$, which is a polynomial in $d$ of degree $k$. No sampling point is needed in all the other subcubes. Due to the known upper bound $N \le C  d  \varepsilon^{-2}$ for the classical star $L_{\infty}$-discrepancy from \eqref{bd:HNWW}, we need at most  $N \le \widetilde C  d^{k+1}  \varepsilon_k^{-2}$ function values in the large subcubes to solve the problem, with another constant $\widetilde C >0$. Hence the number of points increases only polynomially in $d$, as long as $\varepsilon$ is fixed. 
\end{remark}

\subsection{Unanchored discrepancies} 

Unanchored discrepancies are characterized by the fact that all rectangles in the hypercube or on the torus are used as test sets.

\subsubsection{Extreme discrepancy}

The most prominent representative in the class of unanchored discrepancies is the extreme discrepancy, which uses arbitrary axis-parallel boxes as test sets. Let $\DD_d:=\{(\bsx,\bsy): \bsx,\bsy \in [0,1]^d, \bsx \le \bsy\}$, where $\le$ is understood component-wise, and, for $(\bsx,\bsy) \in \DD_d$ let $[\bsx,\bsy)=[x_1,y_1)\times [x_2,y_2) \times \ldots \times [x_d,y_d)$ with volume $\lambda([\bsx,\bsy))=(y_1-x_1)\cdots (y_d-x_d)$. Using as test sets the class $\mathcal{C}=\{[\bsx,\bsy) : (\bsx,\bsy) \in \DD_d\}$, the extreme $L_p$-discrepancy of $\cP$ is defined as
$$L_{p,N}^{\ex}(\cP):=\|\Delta_{\cP}([\cdot,\cdot))\|_{L_p(\DD_d)} \quad \mbox{ for } p \in [1,\infty].$$
The initial extreme $L_p$-discrepancy is
$${\rm disc}_p^{\ex}(0,d)=\left\{ 
\begin{array}{ll}
((p+1)(p+2))^{-d/p} & \mbox{for } p \in [1,\infty),\\
1 & \mbox{for } p=\infty.
\end{array}
\right. 
$$

\noindent{\bf Discrepancy--integration duality:} Define a linear operator $T_d$ on $L_q(\DD_d)$ via 
\begin{equation}\label{eq:rep}
(T_d c)(\bst):=\int_{\DD_d} c(\bsx,\bsy)\, {\bf 1}_{[\bsx,\bsy]}(\bst) \rd (\bsx,\bsy)
\quad\text{for }c\in L_q(\DD_d).
\end{equation}
Define $$F_{d,q}:=\{f \ : \ c\in L_q(\DD_d), \ f=T_dc\}$$ and equip this space with the norm 
\begin{equation}\label{eq:norm-rep}
\|f\|_{F_{d,q},\bx}:=\inf\{\|c\|_{L_q(\DD_d)}:\;c\in L_q(\DD_d), \ f=T_dc\}.
\end{equation}
Then a discrepancy--integration duality is known from \cite[Theorem~2]{NP26a}:

\begin{theorem}\label{thm3}
We have
\[
e(Q_{\cP,\cA};F_{d,q}) \;=\; L_{p,N}^{\ex}(\cP,\cA).
\]
In particular
\begin{equation}\label{eq:did:unanch}
e(N,d;F_{d,q})={\rm disc}_p^{\ex}(N,d).
\end{equation}
\end{theorem}

\noindent{\bf Tractablility vs. curse of dimensionality:}  For $p \in (1,\infty)$ the extreme $L_p$-discrepancy suffers from the curse of dimensionality. This was shown in \cite{NP26a} with the method from Theorem~\ref{thm1} based on the worst-case function $h(x)=1-x^{p+1}-(1-x)^{p+1}$. The case $p=1$ is still open. For $p=\infty$ the extreme $L_{\infty}$-discrepancy is polynomially tractable as shown in \cite{hnww} (see also \cite{Gne05}).

\subsubsection{Periodic discrepancy} 
Another important representative in the class of unanchored discrepancies is the periodic discrepancy, which is also known as wrap-around discrepancy or torus-discrepancy. Here periodic rectangles are used as test sets. For $x,y\in [0,1]$ set
$$ I(x,y):=\begin{cases}
           [x,y) & \text{if $x\leq y$}, \\
           [0,y)\cup [x,1)& \text{if $x>y$,}
          \end{cases}$$
and for $(\bsx,\bsy) \in [0,1]^{2 d}$ we set $B(\bsx,\bsy):=I(x_1,y_1) \times \ldots \times I(x_d,y_d)$ with volume $\lambda(B(\bsx,\bsy))=\prod_{j=1}^d(y_j-x_j \pmod{1})$. Using as test sets the class $\mathcal{C}=\{B(\bsx,\bsy) : (\bsx,\bsy) \in [0,1]^{2 d}\}$, the periodic $L_p$-discrepancy of $\cP$ is defined as
$$L_{p,N}^{\per}(\cP,\cA):=\|\Delta_{\cP,\cA}(B(\cdot,\cdot))\|_{L_p([0,1]^{2 d})} \quad \mbox{ for } p \in [1,\infty].$$
The initial periodic $L_p$-discrepancy is 
$${\rm disc}_p^{\per}(0,d)=\left\{ 
\begin{array}{ll}
(p+1)^{-d/p} & \mbox{for } p \in [1,\infty),\\
1 & \mbox{for } p=\infty,
\end{array}
\right. 
$$
which coincides with the initial star $L_p$-discrepancy.

It is easily seen (or see \cite{KP}) that $L_{p,N}^{{\rm per}}(\cP,\cA) \ge L_{p,N}^{\ex}(\cP,\cA)$ for every $p\ge 1$.\\

\noindent{\bf Discrepancy--integration duality:} First consider the special case $p=q=2$, 
for which a discrepancy--integration duality is well known. Let 
$$
H^1_{\per,2}([0,1]):=\{f:[0,1] \rightarrow \R \ : \  f\, \mbox{abs. cont., $f(0)=f(1)$ and}\, f' \in L_2([0,1])\}
$$
equipped with the norm 
$$
\|f\|_{1,2}^{{\rm per}}:=\left(3 \left(\int_0^1 f(x) \rd x\right)^2+\frac{1}{2} \|f'\|_{L_2}^2  \right)^{1/2}.
$$ 
The $d$-fold tensor product space is $$H_{\per,2}^{\bsone}:=H_{\per,2}^1([0,1]) \otimes \cdots \otimes H_{\per,2}^1([0,1])$$ equipped with the corresponding crossnorm $\|f\|_{d,2}^{{\rm per}}$. Then we know from \cite{HO}:

\begin{theorem}
We have $$e(Q_{\cP,\cA};H_{\per,2}^{\bsone})= L_{2,N}^{\per}(\cP,\cA).$$ In particular, $$e(N,d;H_{\per,2}^{\bsone})={\rm disc}_2^{\per}(N,d).$$ 
\end{theorem}

Now consider the case of general H\"older conjugates  $p,q \in [1,\infty]$. Similarly to \eqref{eq:rep} define a linear operator $T_d^{\per}$ on $L_q([0,1]^{2d})$ via 
\begin{equation*}
(T_d^{\per} c)(\bst):=\int_{[0,1]^{2d}} c(\bsx,\bsy)\, {\bf 1}_{B(\bsx,\bsy)}(\bst) \rd (\bsx,\bsy)
\quad\text{for }c\in L_q([0,1]^{2d}).
\end{equation*}
Define $$F_{d,q}^{\per}:=\{f \ : \ c\in L_q([0,1]^{2d}), \, f=T_d^{\per} c\}$$ and equip this space with the norm 
\begin{equation*}
\|f\|_{F_{d,q}^{\per}}:=\inf\{\|c\|_{L_q([0,1]^{2d})}:\;c\in L_q([0,1]^{2d}), \ f=T_d^{\per}c\}.
\end{equation*}

\begin{remark}\rm
For every $c \in L_q([0,1]^{2d})$ the function $f=T_d^{\per} c$ is well defined a.e., belongs to $L_q([0,1]^d)$, is absolutely continuous in each variable, and is periodic with period 1 in every coordinate. For example, in the univariate case, 
\begin{equation*}
(T_1^{\per}c)(0) = \int_0^1 \int_y^1 c(x,y) \rd x \rd y = \int_0^1 \int_0^x c(x,y) \rd y \rd x = (T_1^{\per}c)(1). 
\end{equation*}
\end{remark}

Now a discrepancy--integration duality is given in the following way:

\begin{theorem}
We have $$e(Q_{\cP,\cA};F_{d,q}^{\per})=L_{p,N}^{\per}(\cP,\cA).$$ In particular, $$e(N,d;F_{d,q}^{\per})={\rm disc}_p^{\per}(N,d).$$
\end{theorem}

\begin{proof}
The result and its proof are new but they are similar to the one of Theorem~\ref{thm2} and the one of Theorem~\ref{thm3} in \cite{NP26a}. Therefore we provide only the cornerstones. Let $L(f):=I(f)-Q_{\cP,\cA}(f)$ and $\mu=\lambda - \sum_{k=1}^N a_k \delta_{\bsx_k}$ 
as in the proof of Theorem~\ref{thm2}. 
Then, for $f =T_d^\per c \in F_{d,q}^{\per}$, we have $L(f) = \int_{[0,1]^d} f(\bst) \rd\mu(\bst)$. From here, using Fubini, we derive 
\begin{equation}\label{fub_per}
L(f) =  \int_{[0,1]^{2d}} c(\bsx,\bsy)\,\Delta_{\cP,\cA}(B(\bsx,\bsy)) \rd (\bsx,\bsy)
\end{equation}
and then Hölder's inequality yields $|L(f)| \le \|c\|_{L_q([0,1]^{2d})}\,L_{p,N}^{\per}(\cP,\cA)$. Taking the infimum over all representations of $f$ yields $|L(f)| \le \|f\|_{F_{d,q}^{\per}}\,L_{p,N}^{\per}(\cP,\cA)$ and hence $$e(Q_{\cP,\cA};F_{d,q}^{\per})\le L_{p,N}^{\per}(\cP,\cA).$$

In order to prove equality we first assume $p,q \in (1, \infty)$. If $\Delta_{\cP,\cA}(B(\cdot,\cdot)) \equiv 0$ then $L_{p,N}^{\per}(\cP,\cA)=0$ and also $L(f)\equiv 0$, 
so equality holds. Assume now $\Delta_{\cP,\cA}(B(\cdot,\cdot))\not\equiv 0$ and define 
\begin{equation}\label{eq:cstar}
c^{\ast}(\bsx,\bsy):=\frac{|\Delta_{\cP,\cA}(B(\bsx,\bsy))|^{p-2} \Delta_{\cP,\cA}(B(\bsx,\bsy))}{\|\Delta_{\cP,\cA}(B(\cdot,\cdot))\|_{L_p([0,1]^{2 d})}^{p-1}} \quad\text{and} \quad f^\ast:=T_d^{\per} c^\ast
\end{equation}
so that $\|c^{\ast}\|_{L_q([0,1]^{2 d})}=1$ and then $\|f^\ast\|_{F_{d,q}^{\per}}\le 1$ by definition of the norm. By~\eqref{fub_per},
\begin{align*}
|L(f^\ast)| = & \frac{1}{\|\Delta_{\cP,\cA}(B(\cdot,\cdot))\|_{L_p([0,1]^{2 d})}^{p-1}}\int_{[0,1]^{2 d}}|\Delta_{\cP,\cA}(B(\bsx,\bsy))|^p \rd (\bsx,\bsy)\\
= & \|\Delta_{\cP,\cA}(B(\cdot,\cdot))\|_{L_p([0,1]^{2 d})}=L_{p,N}^{\per}(\cP,\cA).
\end{align*}
Thus, $e(Q_{\cP,\cA};F_{d,q}^{\per})\ge L_{p,N}^{\per}(\cP,\cA)$. Combining with the upper bound yields the identity.

For the case $p=1$ we take $c^*(\bsx,\bsy) = {\rm sgn} (\Delta_{\cP,\cA} (B(\bsx,\bsy)))$ and obtain 
again $|L(f^\ast)|=\|\Delta_{\cP,\cA}(B(\cdot,\cdot))\|_{L_1([0,1]^{2d})}=L_{1,N}^{\per}(\cP,\cA)$.

For $p=\infty$, $q=1$ and $\varepsilon>0$, define
\[
E_\varepsilon
:=\{(\bsx,\bsy) \in [0,1]^{2 d}: |\Delta_{\cP,\cA}(B(\bsx,\bsy))| \ge L_{\infty,N}^{\per}(\cP,\cA)-\varepsilon\}.
\]
Assume $\lambda(E_\varepsilon)>0$ and define $c_\varepsilon (\bsx, \bsy) :={\rm sgn} (\Delta_{\cP,\cA} (B(\bsx, \bsy)) )\,  {\mathbf 1}_{E_\varepsilon} (\bsx, \bsy)/\lambda(E_\varepsilon)$. Then $\Vert c_\varepsilon\Vert_{L_1([0,1]^{2 d})}=1$ and
\[
\int_{[0,1]^{2 d}}  c_\varepsilon (\bsx,\bsy)\,\Delta_{\cP,\cA}(B(\bsx,\bsy))\rd (\bsx,\bsy )
\ge L_{\infty,N}^{\per}(\cP,\cA)-\varepsilon.
\]
Letting $\varepsilon\to0$ gives $e(Q_{\cP,\cA};F_{d,1})=L_{\infty,N}^{\per}(\cP,\cA)$, as claimed. \qed
\end{proof}

\noindent{\bf Tractablility vs. curse of dimensionality:} It is known from \cite{DHP} that the periodic $L_2$-discrepancy suffers from the curse of dimensionality. We can extend this result to all $p \in (1,\infty)$, at least for non-negative weights $\cA^+$. 

Let $\di_p^{\per,+}(N,d):=\inf_{\cP,\cA^+} L_{p,N}^{{\rm per}}(\cP,\cA^+)$, where the infimum is extended over all $N$-element sets $\cP$ in $[0,1)^d$ and all corresponding non-negative weights $\cA^+$, and let $$N_p^{{\rm per},+}(\varepsilon,d) := \min\left\{N \in \N : {\rm disc}_p^{{\rm per},+}(N,d) \le \frac{\varepsilon}{(p+1)^{d/p}}\right\}.$$

\begin{theorem}\label{thm4}
For every $p$ in $(1,\infty)$ there exists a real $C_p$ that is strictly larger than 1, such that for all $d \in \N$ and all $\varepsilon \in (0,1/2)$ we have $$N_p^{{\rm per},+}(\varepsilon,d) \ge C_p^d \, (1-2 \varepsilon).$$ 
In particular, for all $p$ in $(1,\infty)$ the periodic $L_p$-discrepancy for non-negative weights $\cA^+$ suffers from the curse of dimensionality.
\end{theorem}

\begin{proof}
Let $\cP=\{\bsx_1,\ldots,\bsx_N\}$ be a point set in $[0,1)^d$ and let $\bsdelta \in [0,1)^d$. Then we define $\cP_{\bsdelta}:= \{\{\bsx_1+\bsdelta\},\ldots,\{\bsx_N+\bsdelta\}\}$, where $\{\bsx_k+\bsdelta\}$ denotes component-wise addition modulo one. It is known that
$$L_{p,N}^{{\rm per}}(\cP,\cA)=\left( \int_{[0,1]^d} (L_{p,N}^{\ast}(\cP_{\bsdelta},\cA))^p \rd \bsdelta \right)^{1/p}.$$ 
This is the $L_p$-version of \cite[Proposition~1]{DHP} for $p=2$ whose proof immediately applies to the general case. In particular, for every $N$-element set $\cP$ in $[0,1)^d$ there exists a $\bsdelta_{\ast} \in [0,1]^d$ such that $L_{p,N}^{{\rm per}}(\cP,\cA) \ge L_{p,N}^{\ast}(\cP_{\bsdelta_{\ast}},\cA)$.

Let $\varepsilon \in (0,1/2)$ and assume that $\cP$ is an $N$-element set in $[0,1)^d$ such that $L_{p,N}^{{\rm per}}(\cP,\cA^+) \le \varepsilon \, (p+1)^{-d/p}$. Then there exists a $\bsdelta_{\ast} \in [0,1]^d$ such that $$L_{p,N}^{\ast}(\cP_{\bsdelta_{\ast}},\cA^+) \le \frac{\varepsilon}{(p+1)^{d/p}}.$$ Then \cite[Theorem~3]{NP25} implies that $N \ge C_p^d(1-2\varepsilon)$, where $$C_p:=\left(\frac{1}{2}+\frac{p+1}{p}  \frac{1+2^{p/(p+1)}-2^{1/(p+1)}}{4}\right)^{-1} > 1.$$ Thus, $N_p^{{\rm per},+}(\varepsilon,d)  \ge C_p^d(1-2\varepsilon)$, as claimed. \qed 
\end{proof}

For $p=\infty$ the periodic $L_\infty$-discrepancy is polynomially tractable, more precisely there exists a $C>0$ such that $$N_{\infty}^{{\rm per},\qmc}(\varepsilon,d) \le C d \log (d+1) \varepsilon^{-2} \qquad \mbox{for all } d \in \mathbb{N} \mbox{ and } \varepsilon \in (0,1).$$ This follows from \cite[Theorem~4]{hnww} together with \cite[Theorem~1.1]{GLM}, which states that the VC dimension of the class of periodic intervals is asymptotically $d \log_2 d$.

The case $p=1$ is still open.

\section{Summary, open questions and further remarks}\label{sec:close}

We summarize the situation (all instances for non-negative weights $\cA^+$) in the following table, which also highlights the remaining open questions. For quadrant/anchored discrepancy we assume $a \in (0,1)$, since $a\in \{0,1\}$ is covered by
the  star discrepancy.

$$
\begin{array}{l||c|c|c|c}
\text{discrepancy} & \text{curse} & \text{PT} & \text{open} & \text{source}\\
\hline
L_p^{\ast} & p \in (1,\infty) & p=\infty & p=1 & \cite{NP25}, \cite{hnww}\\
L_p^{\boxplus_a},L_p^{\pitchfork_a} & (p,a) \in [1,\infty) \times (0,1) & - &  -  & \cite{NP25b}, \cite{NW01}\\
                                               & \text{or }  (p,a) = (\infty,\tfrac{1}{2}) & & & \\
L_p^{\ex} & p \in (1,\infty) & p=\infty & p=1 & \cite{NP26a}, \cite{hnww,Gne05}\\
L_p^{\per} & p \in (1,\infty) & p=\infty & p=1 & \text{Thm.~\ref{thm4}}, \cite{hnww,GLM}                                               
\end{array}
$$

Also other norms of discrepancy functions are considered in the literature. 
In \cite{P22} it is shown that the star discrepancy in the BMO seminorm suffers 
from the curse of dimensionality. Curse vs. tractability of star discrepancy in Orlicz norms is studied in \cite{DHPP}. The spherical cap $L_2$-discrepancy is studied in \cite{BDP26}. Also the case of weighted discrepancy is studied intensively, see \cite{NP25c} and the references therein.

It would be interesting to identify other discrepancies/problems where the method from Section~\ref{sec:meth} can be successfully applied.

Furthermore, we mention that for $L_p$-discrepancies, if the H\"older conjugate $q$ is an even number, 
the method of decomposable kernels (see \cite[Chapter~11]{NW10}) can be generalized, 
and in this way lower bounds and statements about the curse of dimensionality can be obtained--even 
for arbitrary quadrature rules. For more details see \cite{NP23,NP25b}.

We stress that it is  a great annoyance that we know so little about the $L_1$-discrepancies. This is the big open question that we  pose in conclusion.

\begin{oq}\rm
Prove or disprove the curse of dimensionality for the star-, extreme-, or periodic $L_1$-discrepancy. 
\end{oq}

\vspace{0.5cm}
\noindent{\bf Author's Address:}

\noindent Erich Novak, Mathematisches Institut, FSU Jena, Inselplatz 5, 07743 Jena, Germany. Email: erich.novak@uni-jena.de\\

\noindent Friedrich Pillichshammer, Institut f\"{u}r Finanzmathematik und Angewandte Zahlentheorie, JKU Linz, Altenbergerstra{\ss}e 69, A-4040 Linz, Austria. Email: friedrich.pillichshammer@jku.at

\end{document}